\documentclass[11pt]{article}

\usepackage{amsmath, amssymb, amsthm}
\usepackage{mathrsfs}
\usepackage{fullpage}
\usepackage{hyperref}
\hypersetup{hidelinks,
  pdftitle={Thresholds for Tic-Tac-Toe on Finite Affine Spaces},
  pdfauthor={Luca Bastioni, Alessandro Giannoni, Javier Lobillo-Olmedo}}
\usepackage{tikz}
\usepackage{enumitem}
\usepackage[dvipsnames]{xcolor}
\usepackage{comment}
\usepackage[hypcap=false]{caption}
\numberwithin{figure}{section}
\newtheorem{definition}{Definition}[section]
\newtheorem{theorem}[definition]{Theorem}
\newtheorem{proposition}[definition]{Proposition}
\newtheorem{lemma}[definition]{Lemma}
\newtheorem{corollary}[definition]{Corollary}
\newtheorem{conjecture}[definition]{Conjecture}

\theoremstyle{definition}
\newtheorem{remark}[definition]{Remark}

\newcommand{\F}{\mathbb{F}}

\newcommand{\AR}{\operatorname{AR}}
\title{Thresholds for Tic-Tac-Toe on Finite Affine Spaces}

\author{
Luca Bastioni\thanks{University of South Florida, 4202 E Fowler Ave, Tampa, FL 33620, USA}
\and
Alessandro Giannoni\thanks{Università degli Studi di Napoli Federico II, Dipartimento di Matematica e Applicazioni}
\and
Javier Lobillo-Olmedo\footnotemark[1]
}

\date{}

\begin{document}

\maketitle

\begin{abstract}
We study the strong positional game in which two players alternately claim
points of $\F_q^m$ and the first to complete an affine $n$-subspace wins.
We distinguish the first winning dimension $\tau(n,q)$ from the eventual
winning threshold $T(n,q)$, without assuming monotonicity in $m$.
Affine Ramsey theory proves finiteness, while the Erd\H{o}s-Selfridge
criterion and explicit geometric strategies give lower bounds.
In the binary case, a quantitative form of the classical quotient argument
yields $T(n,2)\le2^{n-1}+n$, with a corresponding bound for any number
of colors. A plane-counting recursion gives further multicolor and
extremal estimates.
We prove $\tau(2,2)=T(2,2)=4$ and show that every $31$-point subset of
$\F_2^6$ contains an affine $3$-subspace. This gives an elementary proof
of the known equality $\AR_2(3;2)=6$ and implies
$\tau(3,2)=T(3,2)\in\{5,6\}$.
Previously known planar outcomes are recovered within the same framework.
\end{abstract}

\section{Introduction}

In a strong positional game, two players alternately claim elements of a
finite board and compete to occupy an entire member of a prescribed family
of winning sets. Classical Tic-Tac-Toe is a familiar example; general
background can be found in \cite{Beck2008,Beck2005}.
Higher-dimensional versions connect positional games with Ramsey theory,
as illustrated by the Hales-Jewett theorem \cite{HalesJewett}.

We consider the board $\F_q^m$, where $q$ is a prime power and
$1\le n\le m$, with all affine $n$-subspaces as winning sets.
The players claim previously unclaimed points, and the first to complete
one of these subspaces wins. If neither does so before the board is full,
the game is a draw. We call this $(m,n)_q$-Tic-Tac-Toe.
This extends the affine-plane game studied by Carroll and Dougherty
\cite{CarrollDougherty2004}: $(2,1)_q$ is winning for $q\le4$ and drawing
for $q>4$. Danziger, Huggan, Malik, and Marbach
\cite{DanzigerHugganMalikMarbach2022} subsequently gave an explicit,
human-verifiable winning strategy for the plane of order $4$.
An online implementation of the game is available at
\href{https://affinetictactoe.com}{\texttt{affinetictactoe.com}}.

Our question is how the outcome changes with the ambient dimension $m$
when $n$ and $q$ are fixed. Strategy stealing excludes a winning strategy
for the second player, so each game is either a first-player win or a draw.
We distinguish the first winning dimension $\tau(n,q)$ and the eventual
winning threshold $T(n,q)$, defined precisely in
Definition~\ref{def:threshold}, without assuming that the outcome is
monotone in $m$.
We write $\AR_q(n;r)$ for the least ambient dimension in which every
$r$-coloring of the points contains a monochromatic affine $n$-subspace.
The following theorem collects the main bounds.

\begin{theorem}\label{thm:main}
\begin{enumerate}[label=(\roman*),leftmargin=*]
\item For every prime power $q$ and every $n\ge1$,
$$
\tau(n,q)\le T(n,q)\le\AR_q(n;2)<\infty
$$
and
$$
\tau(n,q)\ge n+\left\lceil
\frac{q^n-1-n\log_2\!\bigl(q/(q-1)\bigr)}{(n+1)\log_2 q}
\right\rceil.
$$
Moreover, $\tau(n,q)\ge n+2$ for $n\ge2$.
\item For every $n\ge1$ and $r\ge2$,
$$
\AR_2(n;r)\le\left\lceil2^{n-1}\log_2 r\right\rceil+n.
$$
In particular,
$$
T(n,2)\le2^{n-1}+n=\left(\frac12+o(1)\right)2^n.
$$
For lines, the exact value is $\AR_2(1;r)=\lfloor\log_2 r\rfloor+1$.
\item For the small binary parameters,
$$
\tau(2,2)=T(2,2)=4,\qquad
\AR_2(3;2)=6,\qquad
\tau(3,2)=T(3,2)\in\{5,6\}.
$$
\end{enumerate}
\end{theorem}

The finiteness assertion and the general lower bound are applications of
classical tools: the affine Ramsey theorem of Graham, Leeb, and Rothschild
\cite{GrahamLeebRothschild1972}, and the Erd\H{o}s-Selfridge criterion.
The bound $\tau(n,q)\ge n+2$ also follows from the exact latter criterion,
except when $(n,q)=(2,2)$. We retain an independent geometric proof because
it describes the second player's responses using involutions and parallel
classes of hyperplanes.

The quantitative binary estimate is obtained by keeping track of the loss
at each step of the classical density argument through one-dimensional
quotients. Brown \cite[Theorems~1 and~2]{Brown1985AffineBinary} gives
$\AR_2(n;2^k)\le2^n(k+2)$. Bonin and Qin
\cite{BoninQin2000SizeFunctions} established exponential bounds for
binary affine-subspace-free sets; the quotient method and its quantitative
versions are also developed in
\cite{FredericksonYepremyan2025,KovacsNagy2025}.
The bound in Theorem~\ref{thm:main} improves the leading constant in the
estimate $\AR_2(n;r)\le(\log_2 r)2^n$ recorded in
\cite[Section~7]{FredericksonYepremyan2025}.
The method itself is classical. For fixed $r$, the same leading term
is already implicit in \cite[Remark~4.3]{KovacsNagy2025};
the refinement here is the explicit additive term $n$. Section~\ref{subsec:binary-comparison}
explains the dimension conventions and the explicit error estimate used
here, so that the quantitative refinement is separated from the underlying
existence argument.

The small Ramsey statement $\AR_2(3;2)\le6$ is also known: it is equivalent
to the $3$-normality of every Boolean function on six variables, as recalled
by Charpin \cite[p.~251, Proposition~2]{Charpin2004Normal}.
We give a short proof through plane-free sets and binary linear codes,
together with a matching coloring in dimension $5$. Its consequence for
the game is that the outcome is already monotone for $(n,q)=(3,2)$,
and determining the single game $(5,3)_2$ will determine both thresholds.
The planar results cited above are included as applications of the common
framework, rather than as new determinations.

Table~\ref{tab:small-parameters} collects the best bounds proved below.
The entries $T(4,2)\le12$ and $T(5,2)\le21$ now follow directly from
Theorem~\ref{thm:main}(ii). The two smaller nontrivial binary cases use
the separate arguments in Section~\ref{sec:fourier}.
The plane-counting recursion remains useful for more colors: for example,
it gives $\AR_2(3;3)\le9$, compared with $10$ from the closed formula.
Section~\ref{subsec:ternary-bound} also records the known exponential
upper bound over $\F_3$.

\begin{table}[htbp]
\centering
\renewcommand{\arraystretch}{1.18}
\begin{tabular}{c|c|l}
$n$ & $q$ & Bounds and exact values \\ \hline
$1$ & $2,3$ & $\tau(1,q)=T(1,q)=2$ \\
$1$ & $4$ & $\tau(1,4)=2$, $2\le T(1,4)\le\AR_4(1;2)$ \\
$1$ & $q\ge5$ & $3\le\tau(1,q)\le T(1,q)\le\AR_q(1;2)$ \\ \hline
$2$ & $2$ & $\tau(2,2)=T(2,2)=4$ \\
$2$ & $3$ & $4\le\tau(2,3)\le T(2,3)\le\AR_3(2;2)$ \\
$2$ & $4$ & $5\le\tau(2,4)\le T(2,4)\le\AR_4(2;2)$ \\
$2$ & $5$ & $6\le\tau(2,5)\le T(2,5)\le\AR_5(2;2)$ \\
$2$ & $7$ & $8\le\tau(2,7)\le T(2,7)\le\AR_7(2;2)$ \\ \hline
$3$ & $2$ & $\tau(3,2)=T(3,2)\in\{5,6\}$ \\
$4$ & $2$ & $7\le\tau(4,2)\le T(4,2)\le12$ \\
$5$ & $2$ & $10\le\tau(5,2)\le T(5,2)\le21$
\end{tabular}
\caption{Small-parameter bounds. All field orders are prime powers.}
\label{tab:small-parameters}
\end{table}

Section~2 records the game conventions and strategy stealing.
Section~3 gives the Ramsey upper bound, and Sections~4 and~5 treat
the Erd\H{o}s-Selfridge and geometric drawing strategies.
Section~6 proves the binary estimates and the small binary cases.
Section~7 treats the other fields, including a quantitative ternary bound,
and Section~8 discusses Maker-Breaker games and open problems.

\section{Preliminaries}

Throughout the paper, let $q$ be a prime power and $1 \le n \le m$ be integers. 
We write $V=\F_q^m$ for the $m$-dimensional vector space over the finite field $\F_q$.

\subsection{\texorpdfstring{$(m,n)_q$}{(m,n)q}-Tic-Tac-Toe}

\begin{definition}
The board of $(m,n)_q$-Tic-Tac-Toe is $\F_q^m$. Players $P_1$ and
$P_2$ alternately claim previously unclaimed points. The first to occupy
an entire affine $n$-subspace wins; if the board is filled without a
winning set, the game is a draw.
\end{definition}

A position consists of disjoint sets $(A_1,A_2)$ of claimed points and
the player to move. A strategy prescribes a legal move at each of that
player's turns; it is winning if it guarantees a win and drawing if it
guarantees that the player does not lose. Every play terminates after at
most $q^m$ moves, since each move occupies a new point.

\begin{definition}
An affine subspace of dimension $n$ in $V$ is a subset of the form
$$
a + U = \{ a + u \mid u \in U \},
$$
where $U \le V$ is a linear subspace of dimension $n$ and $a \in V$.
Such a subspace will also be called an \emph{$n$-plane}.
\end{definition}

The affine-geometric version of classical Tic-Tac-Toe can be interpreted as
the game $(2,1)_3$. The board is the affine plane $\F_3^2$, and a player wins
by occupying all the points of an affine line.

In contrast with the usual formulation of the game, where only eight winning
lines are allowed, the affine-geometric version admits all affine lines in
$\F_3^2$, for a total of twelve winning configurations. Thus the classical
game appears as a restricted version of a richer geometric structure. In Figure \ref{fig:F_3^2} we present some lines, including an example not appearing in the original version of the game.

As recalled in Section~\ref{sec:small-cases}, this version is a first-player
win.

\begin{center}
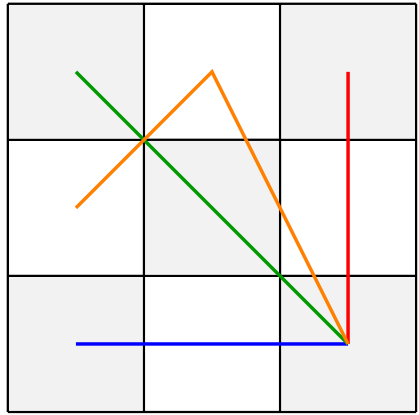

\begin{tikzpicture}[scale=1.8]

  \draw[step=1, thick] (0,0) grid (3,3);

  \foreach \x in {0,1,2}{
    \foreach \y in {0,1,2}{
      \pgfmathparse{mod(\x+\y,2)==0 ? 1 : 0}
      \ifnum\pgfmathresult=1
        \fill[gray!10] (\x,\y) rectangle ++(1,1);
      \fi
    }
  }

  \draw[step=1, thick] (0,0) grid (3,3);

  \draw[line width=1.3pt, blue] (0.5,0.5) -- (2.5,0.5);

  \draw[line width=1.3pt, red] (2.5,0.5) -- (2.5,2.5);

  \draw[line width=1.3pt, green!60!black] (0.5,2.5) -- (2.5,0.5);

  \draw[line width=1.3pt, orange]
      (0.5,1.5) -- (1.5,2.5) -- (2.5,0.5);

  
\end{tikzpicture}
\captionof{figure}{Example of winning lines in $(2,1)_3$ passing through a given square}
\label{fig:F_3^2}
\end{center}

While $(2,1)_3$ recovers and enriches the familiar childhood game, higher-dimensional instances quickly become more intriguing.

Consider for instance the game $(4,2)_3$. 
The board consists of $3^4=81$ points, which can be visualized as a $3\times3$ array of $3\times3$ boards, corresponding to the decomposition
$$
\F_3^4 \cong \F_3^2 \times \F_3^2.
$$
Therefore, coordinates $(x_1,x_2,y_1,y_2)$ refer to the $(y_1,y_2)$-point of the board identified by $(x_1,x_2)$.
A winning configuration is an affine subspace of dimension $2$, hence a set of $3^2=9$ points.
Unlike in the classical case, these winning planes need not be aligned with
the visible $3\times 3$ blocks, as illustrated in Figure \ref{fig:F_3^4}.

Among all small instances, the $(4,2)_3$-game seems particularly well suited
for actual play\footnote{The original idea for this article arose from a game
in this configuration played by the third author and his friend Lukas Mülli at the University of Zürich.}.

\begin{center}
\begin{tikzpicture}[scale=0.84, every node/.style={font=\small}]
  \def\step{5}

  \foreach \xone in {0,1,2}{
    \foreach \xtwo in {0,1,2}{

      \begin{scope}[shift={({\xone*\step},{\xtwo*\step})}]

        \foreach \yone in {0,1,2}{
          \foreach \ytwo in {0,1,2}{
            \pgfmathtruncatemacro{\parity}{mod(\yone+\ytwo,2)}
            \ifnum\parity=0\relax
              \fill[gray!10] (\yone,\ytwo) rectangle ++(1,1);
            \fi
          }
        }

        \draw[step=1, thick] (0,0) grid (3,3);
        \draw[line width=1.1pt] (0,0) rectangle (3,3);

        \node[font=\footnotesize] at (1.5,3.35) {$ (x_1,x_2)=(\xone,\xtwo)$};

        \node[font=\scriptsize, gray!70] at (-0.55,2.5) {$y_2=2$};
        \node[font=\scriptsize, gray!70] at (-0.55,1.5) {$y_2=1$};
        \node[font=\scriptsize, gray!70] at (-0.55,0.5) {$y_2=0$};

        \node[font=\scriptsize, gray!70] at (0.5,-0.28) {$y_1=0$};
        \node[font=\scriptsize, gray!70] at (1.5,-0.28) {$y_1=1$};
        \node[font=\scriptsize, gray!70] at (2.5,-0.28) {$y_1=2$};

      \end{scope}
    }
  }

  \foreach \xone in {0,1,2}{
    \foreach \xtwo in {0,1,2}{
      \foreach \yone in {0,1,2}{
        \foreach \ytwo in {0,1,2}{

          \pgfmathtruncatemacro{\eqA}{mod(\xone+\xtwo+2*\yone+\ytwo-1,3)}
          \pgfmathtruncatemacro{\eqB}{mod(\xone+2*\xtwo+\ytwo,3)}

          \ifnum\eqA=0\relax
            \ifnum\eqB=0\relax
              \draw[line width=1.5pt, blue]
                ({\xone*\step+\yone+0.2},{\xtwo*\step+\ytwo+0.2})
                --
                ({\xone*\step+\yone+0.8},{\xtwo*\step+\ytwo+0.8});
              \draw[line width=1.5pt, blue]
                ({\xone*\step+\yone+0.8},{\xtwo*\step+\ytwo+0.2})
                --
                ({\xone*\step+\yone+0.2},{\xtwo*\step+\ytwo+0.8});
            \fi
          \fi

        }
      }
    }
  }

  \foreach \xone/\xtwo/\yone/\ytwo in {
    0/0/1/2,
    0/2/0/1,
    1/0/2/1,
    1/1/0/0,
    1/2/2/1,
    2/0/1/0,
    2/1/2/2,
    2/2/0/2
  }{
    \draw[line width=1.6pt, red]
      ({\xone*\step+\yone+0.5},{\xtwo*\step+\ytwo+0.5})
      circle (0.35);
  }


\node[align=center, font=\small] at (6.0, 14.1)
  {Blue X's satisfy $x_1+x_2+2y_1+y_2-1=0$ and $x_1+2x_2+y_2=0$ in $\F_3$.};

\end{tikzpicture}

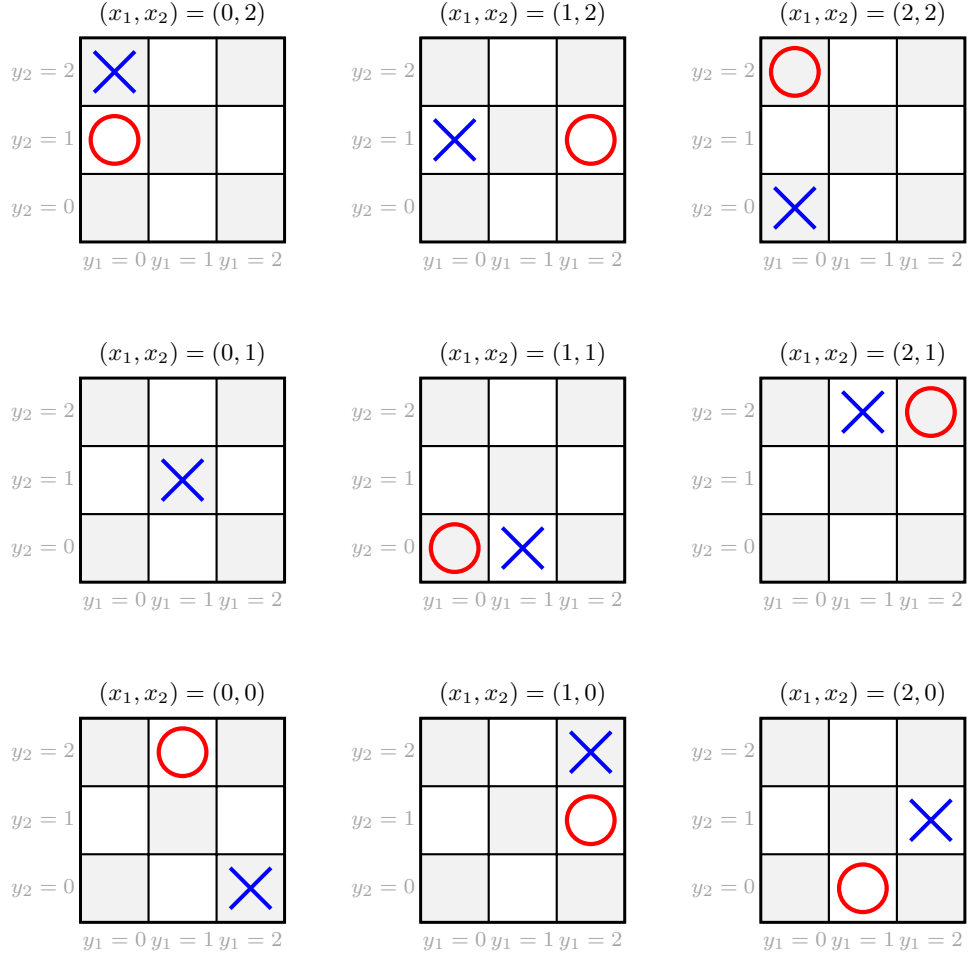
\captionof{figure}{A winning game position for player Blue in $(4,2)_3$ visualized as $\F_3^2\times \F_3^2$.}
\label{fig:F_3^4}
\end{center}

\subsection{Strategies and outcomes}

The game has perfect information and no chance moves. Its game tree is
finite by the termination observation above, so backward induction gives
exactly one of three outcomes: a winning strategy for $P_1$, a winning
strategy for $P_2$, or drawing strategies for both players
\cite[Proposition~99.2]{OsborneRubinstein1994}.
The following standard strategy-stealing argument excludes the second
possibility.

\begin{proposition}\label{prop:P2-cannot-win}
In $(m,n)_q$-Tic-Tac-Toe the second player cannot have a winning strategy.
\end{proposition}

\begin{proof}
Suppose that $P_2$ has a winning strategy $\sigma$. Player $P_1$ first
claims an arbitrary point $x$ and then simulates a game in which the real
$P_2$ is the first player and $P_1$ plays second according to $\sigma$.
The opening move $x$ is initially omitted from the simulated game.

Every real move of $P_2$ is copied into the simulated game. When $\sigma$
prescribes a point not yet owned by the real $P_1$, player $P_1$ claims it.
If that point is already owned by $P_1$, it is recorded as the prescribed
simulated move and $P_1$ claims any other available point in the real game.
Thus the simulated first player's set is exactly the real $P_2$'s set,
whereas the simulated second player's set is contained in the real $P_1$'s
set. These invariants make every copied move legal. If no real point is
available for a replacement move, the simulated game has at most one
unclaimed point, so the winning strategy $\sigma$ must already have
completed a winning set.

The simulated first player cannot win before $\sigma$ does. Consequently,
the real $P_2$ cannot win before the real $P_1$, who owns every point of
the winning set eventually completed by $\sigma$. This contradicts the
assumption that $\sigma$ wins for the real $P_2$.
\end{proof}

\begin{remark}\label{rem:outcome-dichotomy}
Thus every game is either \emph{winning}, meaning that $P_1$ can force a
win, or \emph{drawing}, meaning that both players can avoid defeat.
\end{remark}

\begin{remark}\label{rem:maker-breaker}
In the associated Maker-Breaker game, Maker tries to complete an affine
$n$-subspace and Breaker only tries to prevent this throughout the full
allocation of the board. A winning Breaker strategy gives a drawing
strategy for $P_2$ in the strong game.
A winning Maker strategy alone does not imply a strong-game win, because
it need not defend against the opponent's threats; ordinary $3\times3$
Tic-Tac-Toe is a standard example.
\end{remark}

\subsection{First and eventual winning dimensions}

\begin{definition}\label{def:threshold}
Fix $n\ge1$ and a prime power $q$. Define
$$
\begin{aligned}
\tau(n,q)&=\min\{m\ge n:(m,n)_q\text{ is winning}\},\\
T(n,q)&=\min\{M\ge n:(m,n)_q\text{ is winning for all }m\ge M\},
\end{aligned}
$$
where the minima are over integers, with the convention that the minimum
of an empty set is $\infty$.
We call $\tau(n,q)$ the \emph{first winning dimension} and $T(n,q)$ the
\emph{eventual winning threshold}.
\end{definition}

By definition, $\tau(n,q)\le T(n,q)$. Section~\ref{sec:ramsey-upper-bound}
proves that both parameters are finite.

\begin{remark}\label{rem:monotonicity-caution}
A win in dimension $m$ gives $\tau(n,q)\le m$, but does not by itself
bound $T(n,q)$. A draw there gives $T(n,q)>m$; to conclude
$\tau(n,q)>m$, one must establish draws in every dimension from $n$ to $m$.
A strong-game strategy need not transfer between dimensions, since the
opponent can create threats outside a chosen subspace.
\end{remark}

Since $(n,n)_q$ is drawing, once finiteness is established we have
$$
T(n,q)=1+\max\{m\ge n:(m,n)_q\text{ is drawing}\}.
$$
Consequently, $\tau(n,q)=T(n,q)$ if and only if the outcome is monotone
in the ambient dimension for these parameters.

\section{An affine Ramsey theory upper bound}\label{sec:ramsey-upper-bound}

\begin{definition}
For $n,r\ge1$, let $\AR_q(n;r)$ be the least integer $M$ such that every
$r$-coloring of the points of $\F_q^M$ contains a monochromatic affine
$n$-subspace.
\end{definition}

The affine form of the Graham-Leeb-Rothschild theorem
\cite{GrahamLeebRothschild1972} asserts that this number is finite for every
prime power $q$ and all $n,r\ge1$. The same coloring property then holds
in every dimension $m\ge M$, by restriction to an affine $M$-subspace.
This monotonicity concerns colorings, and does not require monotonicity
of the strong-game outcome.

\begin{proposition}\label{prop:ramsey-game-upper}
For every $n\ge1$ and every prime power $q$,
$$
\tau(n,q)\le T(n,q)\le\AR_q(n;2)<\infty.
$$
\end{proposition}

\begin{proof}
If $m\ge\AR_q(n;2)$, a complete allocation of the board with no winning
set would give a $2$-coloring without a monochromatic affine $n$-subspace.
Thus a draw is impossible, and Proposition~\ref{prop:P2-cannot-win}
implies that $P_1$ has a winning strategy. This holds for every such $m$.
\end{proof}

\section{An Erdős-Selfridge lower bound}\label{sec:erdos-selfridge}

The Erd\H{o}s-Selfridge criterion states that Breaker wins an unbiased
Maker-Breaker game with winning family $\mathcal W$ if
$$
\sum_{W\in\mathcal W}2^{-|W|}<\frac12;
$$
see \cite{ErdosSelfridge1973} and \cite[Theorem~1.4]{Beck2008}.
For the affine $n$-subspaces of $\F_q^m$, every winning set has size $q^n$
and
$$
|\mathcal W|=q^{m-n}\binom{m}{n}_q,
\qquad
\binom{m}{n}_q=\prod_{i=0}^{n-1}\frac{q^m-q^i}{q^n-q^i}.
$$
Indeed, each of the $\binom{m}{n}_q$ linear $n$-subspaces has
$q^{m-n}$ cosets, and each affine subspace has a unique direction.

\begin{theorem}\label{thm:drawing_2qn-1}
If
$$
q^{m-n}\binom{m}{n}_q<2^{q^n-1},
$$
then $(d,n)_q$-Tic-Tac-Toe is drawing for every $n\le d\le m$.
In particular, $m<\tau(n,q)\le T(n,q)$.
\end{theorem}

\begin{proof}
The displayed condition is precisely
$|\mathcal W|2^{-q^n}<1/2$, so Erd\H{o}s-Selfridge gives a winning
Breaker strategy. By Remark~\ref{rem:maker-breaker}, the strong game is
drawing. For $n\le d\le m$, an embedding $\F_q^d\hookrightarrow\F_q^m$
injects the corresponding families of affine $n$-subspaces. Thus the
same criterion holds separately in every such dimension $d$, giving
the claimed initial interval of draws.
\end{proof}

\subsection{An explicit numerical bound}

In order to turn the Erdős-Selfridge condition into an explicit inequality on $m$,
we use a standard estimate for the Gaussian binomial coefficient.

\begin{lemma}\label{lem:gaussian-upper}
For all integers $m\ge n\ge 1$ and all prime powers $q\ge 2$,
$$
\binom{m}{n}_q
\le
\left(\frac{q}{q-1}\right)^n\, q^{n(m-n)}.
$$
Consequently, the number $\mathcal W$ of affine subspaces of dimension $n$ in $\F_q^m$
satisfies
$$
|\mathcal W|
=
q^{m-n}\binom{m}{n}_q
\le
\left(\frac{q}{q-1}\right)^n\, q^{(n+1)(m-n)}.
$$
\end{lemma}

\begin{proof}
The product formula gives
$$
\binom{m}{n}_q=q^{n(m-n)}
\prod_{i=0}^{n-1}\frac{1-q^{i-m}}{1-q^{i-n}}.
$$
Each numerator is at most $1$, and each denominator is at least
$1-q^{-1}$. Multiplication by $q^{m-n}$ gives the second inequality.
\end{proof}

Applying this estimate in Theorem~\ref{thm:drawing_2qn-1}, taking
base-$2$ logarithms and using the integrality of $m$ gives the following
convenient numerical form.

\begin{corollary}\label{cor:ES-final}
For every $n\ge 1$ and every prime power $q\ge 2$,
if
$$
m\le n+\left\lceil
\frac{q^n-1-n\log_2\!\big(\frac{q}{q-1}\big)}{(n+1)\log_2 q}
\right\rceil-1,
$$
then $(m,n)_q$-Tic-Tac-Toe is drawing.
In particular,
$$
T(n,q)\ge\tau(n,q)
\ge n+\left\lceil
\frac{q^n-1-n\log_2\!\big(\frac{q}{q-1}\big)}{(n+1)\log_2 q}
\right\rceil.
$$
\end{corollary}

\section{Pairing strategies and a geometric lower bound}\label{sec:geom-draw}

We give an explicit drawing strategy based on an involution of the board.
An affine $n$-subspace is \emph{unblocked} if it contains no point of $P_2$;
an unblocked subspace is an \emph{immediate threat} if $P_1$ owns all but one
of its points.

\begin{theorem}\label{thm:lower-bound-nplus2}
Assume that $n\ge 2$.
If $m\le n+1$, then $(m,n)_q$-Tic-Tac-Toe is drawing.
In particular,
$$
n+2\le\tau(n,q)\le T(n,q).
$$
\end{theorem}

\begin{proof}
If $m=n$, the only affine $n$-subspace is the whole board, and $P_2$'s
first move prevents $P_1$ from occupying it. It remains to consider $m=n+1$.

First suppose that $(n,q)=(2,2)$. The board has eight points, so $P_1$
has at most four moves. Any three distinct points of $\F_2^3$ lie in a
unique affine $2$-plane. After $P_1$'s third move, $P_2$ claims the fourth
point of that plane if it is available. If it is not available, it is
already owned by $P_2$. Thus $P_1$ cannot complete a plane on his fourth
move, and a draw is guaranteed.

For the remaining cases, $q^{n-1}>2$. We describe the strategy in
three steps.

\emph{1. Pairing by an involution.}
If $q$ is even, fix $v\ne0$ and let $\iota(x)=x+v$; this is a
fixed-point-free involution. If $q$ is odd, let $\iota(x)=-x$; its only
fixed point is $0$.

Before the first immediate threat appears for $P_1$,
player $P_2$ replies to a move $x$ by claiming $\iota(x)$ whenever this
point is distinct from $x$ and available, and otherwise plays any legal
move. After each such reply, every nonfixed point $a$ owned by $P_1$
has $\iota(a)$ owned by $P_2$. This follows inductively from the fact
that $\iota$ is an involution.

\emph{2. The first immediate threat.}
Suppose that $P_1$'s move $x$ creates the first immediate threat $A$, and put $C=\iota(A)$. The affine hyperplanes $A$
and $C$ are parallel. They are distinct: if $A=C$, one of the earlier
nonfixed points of $P_1$ in $A$ would have its image owned by $P_2$ in
$A$, contradicting that $A$ is unblocked. Such a point exists because
there are $q^n-2$ earlier points of $P_1$ in $A$, and at most one fixed
point. In the odd case, $0\notin A$, since an affine subspace containing
$0$ is linear and would satisfy $A=-A$.

It follows that all $q^n-2$ earlier points of $P_1$ in $A$ are nonfixed,
so $P_2$ already owns their images in $C$. Thus at most two points of
$C$ are not owned by $P_2$. Every affine $n$-plane not parallel to $C$
meets $C$ in $q^{n-1}>2$ points and is therefore blocked by $P_2$.
This remains true for the rest of the game.

\emph{3. Subsequent threats.}
At this stage, every unblocked winning hyperplane belongs to the same
parallel class, so a single move can create at most one new immediate
threat. At the first threatening move, every new threat contains $x$,
so there is only one and $P_2$ blocks it. Subsequently, assume that no threat remains
before $P_1$ moves. Any new threat must contain the newly claimed point,
and at most one of the unblocked parallel hyperplanes can contain that
point. Player $P_2$ blocks this threat if it exists, and otherwise makes
any legal move. This maintains the invariant that $P_1$ has no immediate
winning move after $P_2$'s turn. If $P_2$ completes a winning set at any
stage, the strong game simply ends in his favor.

Thus $P_2$ has a strategy that prevents a win for $P_1$. By
Proposition~\ref{prop:P2-cannot-win}, the game is drawing.
\end{proof}

\begin{remark}
For $n=1$ we trivially have $2\le\tau(1,q)\le T(1,q)$: when $m=1$, the only affine
$1$-subspace is the whole board, so $P_1$ cannot occupy all $q$ points before
$P_2$ claims at least one.
\end{remark}

\begin{remark}\label{rem:exact-ES-comparison}
Except for $(n,q)=(2,2)$, the bound in
Theorem~\ref{thm:lower-bound-nplus2} already follows from the exact
Erd\H{o}s-Selfridge criterion. Indeed,
$$
q\binom{n+1}{n}_q
=q\frac{q^{n+1}-1}{q-1}
<2q^{n+1}\le2^{q^n-1}.
$$
For the last inequality, use $\log_2q\le q-1$ and
$q^n\ge(n+1)(q-1)+2$, valid for $n\ge2$ outside the stated exception.
Writing $q=1+t$, the latter follows from
$\binom{n}{2}t^2-t-1\ge0$ by the binomial expansion.
For $(n,q)=(2,2)$ the exact criterion fails, and the eight-point argument
above supplies the missing case. Thus the geometric proof is an independent
constructive description, rather than a general quantitative improvement
over the exact criterion.
\end{remark}

\section{Binary density and Ramsey bounds}\label{sec:fourier}

We first prove the general explicit bound by one-dimensional quotients.
We then count pairs and planes, and derive recursive estimates useful
for other numbers of colors and for extremal cardinalities.
All dimensions in this section are affine dimensions.

\subsection{The explicit bound}

\begin{proposition}\label{prop:asymptotic-half}
For all $n\ge1$ and $r\ge2$,
$$
\AR_2(n;r)\le\left\lceil2^{n-1}\log_2 r\right\rceil+n.
$$
Consequently, $T(n,2)\le2^{n-1}+n$. For $n=1$, the exact value is
$$
\AR_2(1;r)=\lfloor\log_2 r\rfloor+1.
$$
\end{proposition}

\begin{proof}
Let $A$ have $a$ points in a binary space of size $D$. Averaging its
unordered pairs over the $D-1$ nonzero directions gives a direction
$v$ for which the set of cosets of $\langle v\rangle$ wholly contained
in $A$ has cardinality at least
$$
\frac{a(a-1)}{2(D-1)}
=\frac{a^2}{2D}-\frac{a(D-a)}{2D(D-1)}
\ge\frac{a^2}{2D}-\frac12,
$$
since $0\le a\le D$ and $D\ge2$ imply $a(D-a)\le D(D-1)$.

Fix
$$
m\ge\left\lceil2^{n-1}\log_2 r\right\rceil+n
$$
and start with a color class of size at least $2^m/r$. After $j$
successive quotients, let $a_j$ be the size of the quotient set, and put
$$
b_j=2^{m-j}r^{-2^j},\qquad e_j=1-2^{-j}.
$$
We claim that $a_j\ge b_j-e_j$ for $0\le j\le n-1$.
The case $j=0$ is immediate. Our choice of $m$ gives
$b_j\ge2^{n-j}\ge2$, so the lower bound being squared below is
nonnegative. If the claim holds at $j<n-1$, then
$$
\begin{aligned}
a_{j+1}
&\ge\frac{a_j^2}{2\cdot2^{m-j}}-\frac12\\
&\ge b_{j+1}-r^{-2^j}e_j-\frac12\\
&\ge b_{j+1}-1+2^{-(j+1)},
\end{aligned}
$$
where we used $r^{-2^j}\le1/2$ and discarded a nonnegative quadratic
error term. At the last step,
$$
a_{n-1}\ge b_{n-1}-1+2^{-(n-1)}
\ge1+2^{-(n-1)}>1.
$$
Thus the last quotient contains a line. Its full preimage through the
$n-1$ quotient maps is a monochromatic affine $n$-subspace. The game
bound follows from Proposition~\ref{prop:ramsey-game-upper} with $r=2$.

Finally, a binary affine line consists of any two distinct points.
An $r$-coloring forces such a pair exactly when $2^m>r$: otherwise
all points can receive different colors. The least such dimension is
$\lfloor\log_2 r\rfloor+1$.
\end{proof}

\subsection{Comparison with density arguments}\label{subsec:binary-comparison}

Brown's $a(n;r)$ is exactly our $\AR_2(n;r)$; his density parameter
$\bar a(n;\varepsilon)$ is the least dimension forcing an affine
$n$-space in every set of density at least $\varepsilon$.
His estimates for density $2^{-k}$ and $2^k$ colors are
$\bar a(n;2^{-k})\le2^n(k+2)$ and
$a(n;2^k)\le2^n(k+2)$ \cite[Theorems~1 and~2]{Brown1985AffineBinary}.
They use the density argument of Brown and Buhler
\cite{BrownBuhler1982DensityGeometric}.

Bonin and Qin \cite[Lemma~21]{BoninQin2000SizeFunctions} use the
same type of projection argument. In matroid notation an affine geometry
of rank $n+1$ has affine dimension $n$.
In our affine notation their extremal estimate reads
$$
\mathrm{ex}(m,n)<2^{(1-2^{-(n-1)})m+1};
$$
see \cite[Section~2]{FredericksonYepremyan2023} for this explicit
reformulation. A largest color class then gives
$$
\AR_2(n;r)\le\left\lceil2^{n-1}(1+\log_2r)\right\rceil.
$$
The common-difference argument is also presented in
\cite[Theorem~4.1 and Lemma~4.2]{KovacsNagy2025}, with a uniform
sufficient size $(5/2)2^{(1-2^{-(n-1)})m}$.

In the notation of Frederickson and Yepremyan,
$\AR_2(n;r)=R^{\mathrm{aff}}_2(n;r)$, not their projective Ramsey number.
Their Section~7 records
$R^{\mathrm{aff}}_2(n;r)\le(\log_2r)2^n$, while their Lemma~4.1 gives
a general counting-and-lifting framework
\cite{FredericksonYepremyan2025}.
The refinement in Proposition~\ref{prop:asymptotic-half} tracks
cardinalities through these quotients: after $j$ steps, the loss from
$2^{m-j}r^{-2^j}$ is at most $1-2^{-j}$. Stopping at two points after
$n-1$ quotients yields the exponent $2^{n-1}$ on $r$, with an additive
term of $n$ ambient dimensions. Thus the displayed constant improves
that recorded estimate and, for fixed $r$ and large $n$, the displayed
Bonin-Qin consequence, using the established quotient method.

\subsection{Plane-free sets and counting planes}

\begin{definition}
A subset $S\subseteq\F_2^m$ is \emph{$n$-plane-free} if it contains no
affine $n$-subspace. We write $\mathrm{ex}(m,n)$ for its maximum
possible cardinality.
\end{definition}

\begin{lemma}\label{lem:sidon-2plane}
For every $m\ge2$,
$$
\mathrm{ex}(m,2)\le
\left\lfloor\frac{1+\sqrt{2^{m+3}-7}}{2}\right\rfloor.
$$
\end{lemma}

\begin{proof}
Let $S$ be plane-free, with $|S|=s$. The sums $a+b$ over unordered
pairs of distinct points of $S$ are distinct. Indeed, two different pairs
with the same sum cannot share a point; their four distinct points would
sum to zero and form an affine plane. All these sums are nonzero, so
$$
\binom{s}{2}\le2^m-1.
$$
Solving this quadratic inequality proves the claim.
\end{proof}

This is the standard pair-counting bound for binary Sidon sets, where the
Sidon condition is imposed on pairs of distinct points; see
\cite[Section~2]{KovacsNagy2025} for this connection and related bounds.
It has leading constant $\sqrt2$ in the plane-free estimate. The same
pair counts also give a lower bound for the number of planes.

\begin{lemma}\label{rem:count-2planes}
Let $S\subseteq\F_2^m$, with $|S|=s$ and $N=2^m$. The number $N_2(S)$
of affine planes contained in $S$ satisfies
\begin{equation}\label{eq:plane-count}
N_2(S)\ge P(s,m):=
\frac{s(s-1)\bigl(s(s-1)-2(N-1)\bigr)}{24(N-1)}.
\end{equation}
\end{lemma}

\begin{proof}
For $v\ne0$, let $p_v$ be the number of unordered pairs of distinct
points of $S$ with sum $v$. Then
$$
\sum_{v\ne0}p_v=\binom{s}{2},\qquad
\sum_{v\ne0}\binom{p_v}{2}=3N_2(S).
$$
Indeed, two distinct pairs with the same sum are disjoint and form a
plane, while each plane has three partitions into pairs of equal sum.
Cauchy-Schwarz therefore gives
$$
6N_2(S)=\sum_{v\ne0}p_v^2-\binom{s}{2}
\ge\frac{\binom{s}{2}^2}{N-1}-\binom{s}{2},
$$
which is the stated bound.
\end{proof}

This improves the estimate $P_0(s,m)=(s^4/N-3s^2+2s)/24$ obtained
from the zero Fourier coefficient, since
$$
P(s,m)-P_0(s,m)=\frac{s^2(N-s)^2}{24N(N-1)}\ge0.
$$
We use the stronger count $P$ in the recursion below.

\begin{corollary}\label{cor:T22}
$\tau(2,2)=T(2,2)=4$.
\end{corollary}

\begin{proof}
Lemma~\ref{lem:sidon-2plane} gives $\mathrm{ex}(4,2)\le6$.
Thus a color class of at least eight points in $\F_2^4$ contains a plane,
so $T(2,2)\le\AR_2(2;2)\le4$. The reverse inequality for $\tau(2,2)$
follows from Theorem~\ref{thm:lower-bound-nplus2}.
\end{proof}

\subsection{The three-dimensional target}\label{subsec:three-dimensional-target}

\begin{lemma}\label{lem:eight-points-five}
Every set of eight points in $\F_2^5$ contains an affine plane
\cite[Proposition~2.7]{CzerwinskiPott2024}.
\end{lemma}

\begin{proof}
If a plane-free set of eight points existed, translate it to
$\{0,v_1,\ldots,v_7\}$ and let $H$ be the $5\times7$ matrix with
columns $v_1,\ldots,v_7$. Its kernel has dimension at least two.
Every nonzero vector of $\ker H$ has Hamming weight at least five:
weights one and two are excluded by the nonzero, distinct columns;
weights three and four would give four distinct points of the set summing
to zero, using the point $0$ in the weight-three case.
Choose linearly independent $u,v\in\ker H$. Then
$$
\operatorname{wt}(u)+\operatorname{wt}(v)+\operatorname{wt}(u+v)\ge15.
$$
Each of the seven coordinates contributes either zero or two to this sum,
which is therefore at most $14$, a contradiction.
\end{proof}

\begin{proposition}\label{prop:AR-three}
Every subset of $\F_2^6$ of size at least $31$ contains an affine
$3$-subspace. In particular, $\mathrm{ex}(6,3)\le30$. Moreover,
$\AR_2(3;2)=6$.
\end{proposition}

\begin{proof}
It suffices to consider a set $A$ of size $31$. Since
$$
\binom{31}{2}=465>7(2^6-1)=441,
$$
some direction $\langle v\rangle$ contains at least eight lines wholly
contained in $A$.
Their images in the quotient $\F_2^6/\langle v\rangle\cong\F_2^5$
contain a plane by Lemma~\ref{lem:eight-points-five}. Its full preimage
is an affine $3$-subspace contained in $A$. A largest color class in
any two-coloring has at least $32$ points, proving $\AR_2(3;2)\le6$.

For the lower bound, color $\F_2^5$ according to whether the Hamming
weight is at most two or at least three. If $a+U$ is an affine
$3$-subspace, choose three coordinates on which the projection of $U$
has rank three. The projection of $a+U$ onto these coordinates is all
of $\F_2^3$. It therefore contains a point with three zero coordinates,
whose weight is at most two, and a point with three one coordinates,
whose weight is at least three. Thus no affine $3$-subspace is
monochromatic, proving the reverse inequality.
\end{proof}

The upper Ramsey bound is already implicit in the classical normality
result recalled in \cite[p.~251, Proposition~2]{Charpin2004Normal}.
The preceding proof gives the stronger density statement for each
$31$-point set directly, without using a classification of Boolean
functions.

\begin{corollary}\label{cor:tauT32}
$\tau(3,2)=T(3,2)\in\{5,6\}$.
\end{corollary}

\begin{proof}
Theorem~\ref{thm:lower-bound-nplus2} gives draws in dimensions $3$ and
$4$, while Proposition~\ref{prop:AR-three} and
Proposition~\ref{prop:ramsey-game-upper} give wins in every dimension
at least $6$. If $(5,3)_2$ is winning, both thresholds are $5$;
otherwise both are $6$.
\end{proof}

\subsection{A plane-counting recursion}\label{subsec:recursion}

For a $2$-dimensional linear subspace $U\le\F_2^m$, define
$$
Q_U(S)=\{\bar a\in\F_2^m/U:a+U\subseteq S\}.
$$
Writing $G_m=\binom{m}{2}_2=(2^m-1)(2^m-2)/6$, the count of planes
by their directions gives
$$
\sum_U|Q_U(S)|=N_2(S).
$$
Thus some $U$ satisfies $|Q_U(S)|\ge P(|S|,m)/G_m$.
If $Q_U(S)$ contains an affine $(n-2)$-subspace, its full preimage is
an affine $n$-subspace contained in $S$. These two observations give
the following recursion.

\begin{definition}\label{def:f-recursive}
For integers $m\ge n\ge1$, define $f(n,m)$ by
$$
f(1,m)=2,\qquad
f(2,m)=\left\lfloor\frac{1+\sqrt{2^{m+3}-7}}2\right\rfloor+1.
$$
For $n\ge3$, let $f(n,m)$ be the smallest integer
$s\in\{1,\ldots,2^m+1\}$ such that
$$
\frac{P(t,m)}{G_m}>f(n-2,m-2)-1
\quad\text{for every integer }s\le t\le2^m,
$$
where $P$ is defined in \eqref{eq:plane-count}. The condition is empty
when $s=2^m+1$.
\end{definition}

\begin{theorem}\label{thm:n-plane-free-bound}
For all $m\ge n\ge1$, every subset of $\F_2^m$ of size at least
$f(n,m)$ contains an affine $n$-subspace. Equivalently,
$$
\mathrm{ex}(m,n)<f(n,m).
$$
\end{theorem}

\begin{proof}
We induct on $n$, simultaneously for all $m\ge n$. The case $n=1$
is immediate, and $n=2$ is Lemma~\ref{lem:sidon-2plane}.
Let $n\ge3$ and $|S|=s\ge f(n,m)$. By the definition and the
averaging identity above, some $U$ satisfies
$$
|Q_U(S)|\ge\frac{P(s,m)}{G_m}>f(n-2,m-2)-1.
$$
Since its cardinality is an integer, $Q_U(S)$ has at least
$f(n-2,m-2)$ points. The inductive hypothesis provides an affine
$(n-2)$-subspace in $Q_U(S)$, whose preimage proves the claim.
\end{proof}

The following uniform bound ensures that the recursion eventually
certifies every fixed positive density. Its extremal exponent is the
classical Bonin-Qin exponent recalled in
Section~\ref{subsec:binary-comparison}.

\begin{proposition}\label{prop:exponent}
Let $\alpha(n)=1-2^{-(n-1)}$. For every $m\ge n\ge1$,
$$
f(n,m)\le\left\lceil2\cdot2^{\alpha(n)m}\right\rceil.
$$
\end{proposition}

\begin{proof}
We induct on $n$. The definitions give the claim for $n=1,2$.
For $n\ge3$, put $s_0=\lceil2\cdot2^{\alpha(n)m}\rceil$.
If $s_0\ge2^m+1$, the claim is immediate. Otherwise let
$s_0\le t\le2^m$. Since $\alpha(n)\ge3/4$ and $m\ge3$,
$t\ge\sqrt6\,2^{m/2}$. Using $P(t,m)\ge P_0(t,m)$ gives
$$
P(t,m)\ge\frac1{24}\left(\frac{t^4}{2^m}-3t^2+2t\right)
\ge\frac{t^4}{48\cdot2^m}.
$$
Together with $G_m\le2^{2m}/6$ and
$4\alpha(n)-3=\alpha(n-2)$, this yields
$$
\frac{P(t,m)}{G_m}\ge\frac{t^4}{8\cdot2^{3m}}
\ge2\cdot2^{\alpha(n-2)m}.
$$
By induction,
$$
f(n-2,m-2)-1<2\cdot2^{\alpha(n-2)(m-2)}
\le2\cdot2^{\alpha(n-2)m}.
$$
Hence every such $t$ satisfies the defining inequality for $f(n,m)$,
and $f(n,m)\le s_0$.
\end{proof}

\begin{corollary}\label{cor:AR2-general}
For every $n\ge1$ and $r\ge2$,
$$
\AR_2(n;r)\le
\min\left\{m\ge n:f(n,m)\le\left\lceil\frac{2^m}{r}\right\rceil\right\}.
$$
\end{corollary}

\begin{proof}
A largest color class has at least $\lceil2^m/r\rceil$ points, so
Theorem~\ref{thm:n-plane-free-bound} applies. The set defining the
minimum is nonempty by Proposition~\ref{prop:exponent}.
\end{proof}

For $1\le n\le5$ and two colors, the least dimensions certified by the
recursion are $2,4,7,12,21$, respectively. The corresponding values are
$$
\begin{aligned}
f(1,2)&=2, & f(2,4)&=7, & f(3,7)&=56,\\
f(4,12)&=1877, & f(5,21)&=1004208.
\end{aligned}
$$
These dimensions already follow from
Proposition~\ref{prop:asymptotic-half}; the separate argument of
Proposition~\ref{prop:AR-three} gives the improvement to $6$ when $n=3$.
The recursion does, however, improve the closed formula for other
numbers of colors. For example,
$$
f(3,9)=155\le\left\lceil\frac{512}{3}\right\rceil,
\qquad
f(4,16)=21249\le\left\lceil\frac{65536}{3}\right\rceil,
$$
so
$$
\AR_2(3;3)\le9,\qquad \AR_2(4;3)\le16,
$$
compared with $10$ and $17$, respectively, from
Proposition~\ref{prop:asymptotic-half}. It also gives, for instance,
$\mathrm{ex}(9,3)\le154$.
The displayed values follow by direct calculation from
Definition~\ref{def:f-recursive}.

In fact, $\AR_2(3;3)\le8$. A largest color class in $\F_2^8$
has at least $86$ points, and
$\binom{86}{2}>12(2^8-1)$.
Thus some direction has at least $13$ lines wholly contained in
that class. Their images in the seven-dimensional quotient contain
an affine plane by \cite[Proposition~2.8]{CzerwinskiPott2024}.
Lifting this plane gives a monochromatic affine $3$-subspace.
\section{Other fields and the ternary bound}\label{sec:small-cases}

We explain the remaining entries of Table~\ref{tab:small-parameters}.
The outcomes for affine planes are known. We recall their consequences
for the two thresholds and then give a quantitative bound over $\F_3$.

\subsection{A blocking-set criterion}

An affine $n$-blocking set in $\F_q^m$ meets every affine $n$-subspace.
Let $\beta(m,n,q)$ be its minimum cardinality.

\begin{proposition}\label{prop:blocking-size-criterion}
If $\beta(m,n,q)>\lfloor q^m/2\rfloor$, then $(m,n)_q$ is winning.
\end{proposition}

\begin{proof}
In a full allocation, $P_2$ owns at most $\lfloor q^m/2\rfloor$ points.
His set cannot be blocking, so its complement, owned by $P_1$, contains
a winning set. Hence a draw is impossible, and strategy stealing proves
the claim.
\end{proof}

\subsection{Lines in affine planes}

Carroll and Dougherty \cite{CarrollDougherty2004} proved that $(2,1)_q$
is winning for $q\le4$ and drawing for $q>4$.
For the explicit strategy when $q=4$, see
\cite{DanzigerHugganMalikMarbach2022}.
Since $(1,1)_q$ is always drawing, this gives
$$
\tau(1,2)=\tau(1,3)=\tau(1,4)=2,
\qquad \tau(1,q)\ge3\quad(q\ge5).
$$

For $q=2$, any two distinct points form a line, so $P_1$ wins on his
second move in every dimension $m\ge2$. Thus $T(1,2)=2$.
For $q=3$, the Brouwer-Schrijver bound
\cite{BrouwerSchrijver1978} gives $\beta(2,1,3)\ge2\cdot3-1=5$.
If a set of five points in $\F_3^2$ contained no line, its complement
would be a blocking set of size four, a contradiction. Thus every
$2$-coloring of the nine points has a monochromatic line, giving
$T(1,3)\le\AR_3(1;2)\le2$. Consequently,
$$
\tau(1,2)=T(1,2)=2,\qquad \tau(1,3)=T(1,3)=2.
$$

For $q=4$, the planar strategy determines $\tau(1,4)=2$, with the
caution of Remark~\ref{rem:monotonicity-caution} concerning $T(1,4)$.
For $q\ge5$, the draws in dimensions one and two give the lower bound
in Table~\ref{tab:small-parameters}, and the eventual upper bound follows
from Proposition~\ref{prop:ramsey-game-upper}.

\subsection{Two-dimensional targets}

Theorem~\ref{thm:lower-bound-nplus2} and
Corollary~\ref{cor:ES-final} give
$$
\tau(2,q)\ge\max\left\{4,\quad
2+\left\lceil
\frac{q^2-1-2\log_2\!\bigl(q/(q-1)\bigr)}{3\log_2q}
\right\rceil\right\}.
$$
For $q=3,4,5,7$, this gives $4,5,6,8$, respectively, as listed in
Table~\ref{tab:small-parameters}. These lower bounds can also be checked
directly from the exact criterion of Theorem~\ref{thm:drawing_2qn-1}.
For $q=2$, Corollary~\ref{cor:T22} determines both thresholds.

\subsection{A quantitative bound over \texorpdfstring{$\F_3$}{F3}}\label{subsec:ternary-bound}

Let $\sigma_3\approx13.901$ be the constant in
\cite[Lemma~1.10(c)]{FredericksonYepremyan2025}: every
$A\subseteq\F_3^d$ of density $\delta$ contains at least
$\delta^{\sigma_3}3^{2d}$ ordered triples $(x,x+v,x+2v)$, allowing
$v=0$. This is the stated weak Sidorenko property of a ternary affine line.
Their Theorem~4.2 then gives
$$
\mathrm{ex}_{\mathrm{aff}}(m,\F_3^n)
<3^{m-(m-n)/(\sigma_3^n-n)},
$$
where the left-hand side is the largest size of an affine
$n$-subspace-free subset of $\F_3^m$.
A color class of size at least $3^m/2$ therefore contains an affine
$n$-subspace whenever $m\ge n+(\sigma_3^n-n)\log_3 2$. Consequently,
$$
T(n,3)\le\AR_3(n;2)
\le\left\lceil n+(\sigma_3^n-n)\log_3 2\right\rceil
=O(\sigma_3^n).
$$
This is a consequence of the cited extremal theorem, consistent with
the affine Ramsey bounds in Section~7 of the same source.

\section{Concluding remarks and open problems}

\subsection{Weak win and the Maker-Breaker threshold}

We now compare the strong game with the Maker-Breaker game defined in
Remark~\ref{rem:maker-breaker}.

Beck's criterion for Maker-Breaker games~\cite[Theorem~1.2]{Beck2008}
provides a sufficient condition for Maker to win.
Applied to our setting it yields the following.

\begin{proposition}\label{prop:beck-MB}
If
$$
q^m - 1 > q^n(q^n-1)\,2^{q^n-3},
$$
then Maker has a winning strategy in the Maker-Breaker game on $(\F_q^m, \mathcal{W})$,
where $\mathcal{W}$ is the family of all affine $n$-subspaces.
In particular, the Maker-Breaker game is a Maker win whenever
$$
m > 2n + q^n \log_q 2.
$$
\end{proposition}

\begin{proof}
The hypergraph $(\F_q^m, \mathcal W)$ is $q^n$-uniform with $|\mathcal W|/|V| = q^{-n}\binom{m}{n}_q$
and pair-degree $\Delta_2(\mathcal W) = \binom{m-1}{n-1}_q$.
The identity
$$
\frac{\binom{m}{n}_q}{\binom{m-1}{n-1}_q}
=\frac{q^m-1}{q^n-1}
$$
shows that Beck's condition
$|\mathcal W|/|V|>2^{q^n-3}\,\Delta_2(\mathcal W)$ is equivalent to
$q^m-1>q^n(q^n-1)\,2^{q^n-3}$.
The simplified form follows from the bound $q^n(q^n-1)\,2^{q^n-3} < q^{2n}\,2^{q^n}$.
\end{proof}

By Remark~\ref{rem:maker-breaker}, Proposition~\ref{prop:beck-MB}
does not by itself bound either strong-game threshold.

Nevertheless, the Maker-Breaker threshold is suggestive.
Writing $T^{\mathrm{MB}}(n,q)$ for the smallest $m$ such that Maker wins
the Maker-Breaker game on $(\F_q^m,\mathcal W)$, we have
$$
T^{\mathrm{MB}}(n,q)\le2n+q^n\log_q2+1,
\qquad T^{\mathrm{MB}}(n,q)\le\tau(n,q)\le T(n,q).
$$
Here monotonicity in the ambient dimension is valid: Maker can restrict
his play to a winning subspace and treat moves outside it as passes by
Breaker, who has no winning threats. More explicitly, such a pass can
be simulated by an arbitrary additional Breaker move inside the
subspace. The simulated Breaker set then contains the real one there,
so Maker's prescribed moves remain available. If a later real move
lands on an already simulated Breaker point, another unused point is
assigned to Breaker in the simulation. A simulated Maker win is
therefore a real Maker win.

Together with the Erd\H{o}s-Selfridge lower bound, the upper bound above
places the Maker-Breaker threshold between orders $q^n/n$ and $q^n$,
with constants depending on $q$.

For the strong game, Proposition~\ref{prop:asymptotic-half} gives
$T(n,2)\le2^{n-1}+n$.
Together with the Erd\H{o}s-Selfridge lower bound, this places both
$\tau(n,2)$ and $T(n,2)$ between roughly $2^n/(n+1)$ and
$(\frac12+o(1))2^n$. In particular,
$$
\lim_{n\to\infty}\frac{\log_2\tau(n,2)}{n}
=\lim_{n\to\infty}\frac{\log_2 T(n,2)}{n}=1.
$$
This motivates the following conjecture about the eventual threshold,
which would also give the same upper bound for the first winning dimension.

\begin{conjecture}\label{conj:exponential}
For every prime power $q$, there exists a constant $C_q>0$ such that
$$
T(n,q)\le C_q q^n
$$
for every $n\ge 1$.
\end{conjecture}

For $q=3$, Section~\ref{subsec:ternary-bound} already gives an
exponential bound with base $\sigma_3\approx13.901$; the conjecture asks
for the smaller base $3$. For general $q\ge3$, stronger extremal
bounds at density $1/2$ or new game-specific strategies could provide
a route to the proposed base $q$.

\subsection{Open problems}

We collect some open problems arising from this work.

\begin{enumerate}[label=(\roman*)]
\item \emph{Monotonicity in the ambient dimension.}
Is $\tau(n,q)=T(n,q)$ for every $n$ and $q$? Equivalently, does a win
in dimension $m$ always imply a win in dimension $m+1$? If not, determine
parameters for which a winning dimension is followed by a drawing one.
The equality is already established for $(n,q)=(3,2)$, even though its
common value remains undetermined. More generally, can one give useful
sufficient conditions for transferring a strong-game winning strategy to
a larger affine space?

\item \emph{The parameters $\tau(3,2)$ and $T(3,2)$.}
Corollary~\ref{cor:tauT32} shows that
$\tau(3,2)=T(3,2)\in\{5,6\}$. Determine the outcome of $(5,3)_2$:
a win would make both values $5$, while a draw would make both $6$.
An explicit strategy would be particularly useful. A computational
resolution should include a reproducible search and independently
checkable verification of the outcome.

\item \emph{Direct upper bounds for $q\ge 3$.}
General affine Ramsey theory gives upper bounds on $T(n,q)$ for every prime
power $q$, and quantitative estimates for vector-space Ramsey numbers provide
sharper estimates in several cases; see \cite{FredericksonYepremyan2025}.
Can direct game-specific arguments improve these upper bounds?
For $q=3$, a concrete objective is to reduce the exponential base in
Section~\ref{subsec:ternary-bound}, ideally to the base $3$ proposed in
Conjecture~\ref{conj:exponential}.

This is closely related to bounding the maximum size of subsets of $\F_q^m$
that contain no affine subspace of dimension~$n$.

\item \emph{Exact thresholds.}
Determine $\tau(n,q)$ and $T(n,q)$ exactly for further pairs $(n,q)$.
The exact common values established here are
$\tau(1,2)=T(1,2)=2$, $\tau(1,3)=T(1,3)=2$, and
$\tau(2,2)=T(2,2)=4$. We also have $\tau(1,4)=2$; does $T(1,4)=2$?

\item \emph{Online affine Ramsey games.}
The Ramsey upper bounds used in this paper are offline: they rely on the fact
that every sufficiently large complete $2$-coloring of $\F_q^m$ contains a
monochromatic affine $n$-subspace. A natural related problem is to formulate
and study an online affine Ramsey game, in the spirit of Builder-Painter
online Ramsey games \cite{Conlon2009,ConlonFoxGrinshpunHe2019}, where points or affine subspaces are revealed
sequentially and one asks how efficiently a monochromatic affine $n$-subspace
can be forced.

How do the corresponding online affine Ramsey numbers compare with the
parameters $\tau(n,q)$ and $T(n,q)$ studied here?
\end{enumerate}

\section*{Acknowledgements}

The authors thank Eion Mulrenin for useful comments and suggestions, in
particular for pointing out the connection with affine and vector-space Ramsey
theorems and with online Ramsey games.

\section*{Declarations}

\paragraph{Conflicts of interest.}
The authors have no conflicts of interest to declare that are relevant to the
content of this article.

\bibliographystyle{plain}
\bibliography{references}

\end{document}